\documentclass[11pt, psamsfonts,reqno]  {amsart}

\usepackage{amssymb,amsfonts,amsmath,graphicx,lineno}
\usepackage{mathrsfs}
\usepackage{anysize}
\usepackage[colorlinks=true, citecolor=cyan, urlcolor=black, linkcolor=red]  {hyperref}
\usepackage{url}
\usepackage{color}
\usepackage[all]  {xy}
\usepackage{tikz}
\usetikzlibrary{shapes, arrows}
\usepackage{pgfplots}
\tikzset{>= angle 60}

\usepackage{tikz-cd} 

\usepackage{graphics}
\usepackage{amsmath}
\usepackage{amsthm}
\usepackage{amssymb}
\usepackage{amscd}
\usepackage{amsfonts}
\usepackage{amsbsy}
\usepackage{epsfig}
\usepackage{color}

\newtheorem{theorem}{Theorem}[section]
\newtheorem{proposition}[theorem] {Proposition}
\newtheorem{definition}[theorem]{Definition}
\newtheorem{corollary}[theorem]{Corollary}
\newtheorem{lemma}[theorem]{Lemma}
\newtheorem{remark}[theorem]{Remark}

\newcommand{\id}{\mathrm{id}}

\def\M{\mathcal{M}}

\def\P{\mathcal{P}}

\def\QQ{\mathcal{Q}}
\def\p{\boldsymbol{p}}

\def\bdelta{\boldsymbol{a}}

\def\R{{\mathbb R}}

\def\Z{{\mathbb Z}}

\def\lf{\lfloor}
\def\rf{\rfloor}

\begin{document}

\title{Rotation number of 2-interval piecewise affine maps}

\maketitle

\centerline{ Jos\'e Pedro Gaiv\~ao, Michel Laurent and  Arnaldo Nogueira}

\begin{abstract}
We study  maps of the unit interval whose graph is made up of two increasing segments and which are injective in an extended sense. Such maps $f_{\p}$ are parametrized by a quintuple $\p$ of real numbers satisfying  inequations. Viewing $f_{\p}$ as a circle map, we show that it has a rotation number $\rho(f_{\p})$ and we compute $\rho(f_{\p})$ as a function of $\p$ in terms of Hecke-Mahler series.
As a corollary, we prove that $\rho(f_{\p})$ is a rational number when the components of $\p$ are algebraic numbers. 
\end{abstract}

\tableofcontents

\footnote{\footnote \rm 2010 {\it Mathematics Subject Classification:}   
11J91,   37E05. }
 
 \section{Introduction}\label{sec:intro}
Let $\R/\Z$ denote the unit circle and $f\colon \R/\Z\to\R/\Z$ be an  orientation-preserving circle homeomorphism. Any continuous lift $F\colon \R\to\R$ of $f$ is strictly increasing and $F-\id$ is $one$-periodic. In order to study the dynamics of orientation preserving circle homeomorphisms, Poincar\'e introduced an invariant quantity $\rho(f)\in [0,1)$, known as the {\it rotation number} of $f$ that measures the average rotation along any orbit of $f$. Given any continuous lift $F\colon \R\to\R$ of $f$ and $x\in\R$, the rotation number of $f$ is defined as 
$$ 
\rho(f):=\lim_{n \to \infty}\frac{F^n(x)}{n}\pmod 1.
$$ 
This limit exists and is independent of $x$ and the lift $F$. Moreover, $\rho(f)$ is rational if and only if $f$ 
has a periodic point.


The theory of  rotation number was extended to orientation-preserving circle maps which are not continuous, neither surjective, in particular by Rhodes and Thompson  \cite{RT86,RT91}. Let $\M$ denote the set of circle maps $f$ whose lifts $F$ are strictly increasing and $F-\id$ is $one$-periodic. 
Rhodes and Thompson  have shown that the rotation number is well defined for all maps $f\in \M$ and it varies continuously as a function of $f$ on any continuous one parameter family contained in $\M$.  

By identifying the unit interval $I:=[0,1)$ with $\R/\Z$ through the canonical bijection $I\hookrightarrow \R \to \R/\Z$, we may view any circle map in $\M$ as an orientation preserving injective map of $I$. Through this identification, we associate a rotation number to any orientation preserving injective  map. 


The study of the rotation number of injective piecewise affine increasing contractions with only one discontinuity point on $\R/\Z$ and a unique slope was introduced in Canaiello's neuronic equations which are $2$-interval piecewise contractions, see \cite{H15} for a discussion of the topic. 
The dynamics of piecewise contractions has been studied by many authors, in particular \cite{CGMU16,G20,GCM20,NPR14,NP15,NPR18}.
  A detailed study of contracted rotations (meaning that the two branches have equal slopes) can be found in  \cite{NS72,DH87,FC91,Bu93,BuC99,Bu04,H15,LN18,N18,JO19,BS20,CCG21,BuKLN21,GN22, ABP22}. In \cite{LN18},  Laurent and Nogueira were the first to relate transcendental properties of  the parameters of these maps  to the irrationality of their rotation number.  Later in \cite{LN21},  Laurent and Nogueira  extended their work to  allow  injective maps with two different slopes and a single discontinuity point. However, they also studied the dynamics of maps  which are not piecewise contractions, in particular see  \cite[Corollary, page 36]{LN21}, where the rotation number of a $2$-interval piecewise affine circle homeomorphism is obtained. 

This family of maps is also known in the literature as contractive piecewise linear Lorenz maps.   It is sometimes claimed, see for instance \cite[page 237]{CD15},  that the dynamics of such a  contractive map is trivial, meaning that all its orbits converge to a periodic cycle. This is almost true, but not always. In fact, for an uncountable set of parameters, the corresponding maps have singularly continuous invariant probability measures and all the orbits converge to a Cantor set sharing fine arithmetical properties  which are investigated in \cite{Bu04} and \cite{BuKLN21}.

In the present work, we extend the framework of \cite{LN21} by allowing an additional discontinuity to the circle map $f$ at the break point (denoted $\eta$ below) and even  allowing $f$ to be non-injective in some cases.  Moreover, the family of maps we consider are not necessarily piecewise contractions, i.e., one of its branches may expand.

The article is organized as follows.  In Section~\ref{sec:setting} we introduce our family of maps whose set of parameters  is  described in Section~\ref{sec:desc}.  In Section~\ref{sec:rotation} we prove that the maps we consider have a well-defined rotation number. Our main results will be stated in Section~\ref{sec:results}.  Theorem~\ref{maintheorem} describes the rotation number of  $f=f_{\p}$ as a function of the parameter $\p$  and Theorem~\ref{conjugacy} makes explicit the semi-conjugacy of  $f$ to a rotation following the approach given in  \cite{Co99}.  In order to prove our main results  the dynamics of our map is reduced  to that of a map on an invariant interval. Up to an isomorphism, this map belongs to the family of  maps that  have been studied  in \cite{LN21}. 
In Corollary~\ref{maincorollary}, using transcendence results on Hecke-Mahler series, it is proved that the rotation number takes a rational value when all  the components of $\p$ are algebraic numbers.
 In Section~\ref{sec:reduction}, we explain our strategy for the proof of the results in Section~\ref{sec:results}. The properties  of two functions $\bdelta$ and  $\phi$ describing the dynamics of $f$ are displayed in  Section~\ref{sec:properties}. The proofs of the main results from Section \ref{sec:results} are given in Section~\ref{sec:proofs}.

\section{The setting}\label{sec:setting}
Let $I = [0,1)$ be the unit interval identified with the circle $\R/ \Z$ as in Section~\ref{sec:intro}. We are concerned with the dynamics of maps $f : I \to I$ as pictured in Figure \ref{fig:f} below. Namely, the graph of $f$ is made up with two increasing segments such that $f$ has no fixed point  and is injective when restricted to a certain invariant subinterval of $I$. Such a map $f$ has a {\it rotation number} $\rho(f)\in I$ and is semi-conjugated to the rotation $R_\rho : x\to x+ \rho \pmod{1}$. 
Our goal is to make explicit these two assertions in terms of the parameters defining $f$. The precise statements are given in Theorem  \ref{maintheorem} and \ref{conjugacy} below.

\begin{figure}[ht] 
\begin{tikzpicture}
\draw (0,0) -- (5,0);
\draw (0,0) -- (0,5);
\draw[thick]   (0,3.5) -- (2.5,4.5);
\draw[thick]  (2.5,1) -- (5,4);
\node   at (-0.3,-0.3) {$0$};
\draw[dashed] (0,1) -- (2.5,1);
\draw[dashed] (0,5) -- (5,5);
\draw[dashed] (5,0) -- (5,5);
\draw[dashed] (0,4.5) -- (2.5,4.5);
\node at (1,4.2) {$\lambda$};
\node at (3.6,2.8) {$\lambda\mu$};
\node at (-0.3,4.5) {$b$};
\node   at (-0.3,1) {$c$};
\node at (2.5,-0.6) {$\eta = \frac{b-a}{\lambda}$};
\node   at (-0.3,5) {$1$};
\node   at (-0.3,3.5) {$a$};
\node   at (5,-0.4) {$1$};
\draw[dashed]  (2.5,0) -- (2.5,4.5);
\end{tikzpicture}
\caption{ A plot of $f_{p}$}
\label{fig:f}
\end{figure}
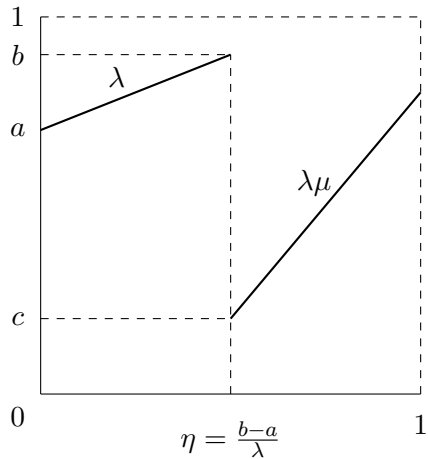

We parametrize those maps $f$ by a quintuple $\p=(\lambda, \mu ,a, b, c)$ as follows. 
\begin{definition} 
 Let us denote by $\P$ the set of quintuple parameters  $\p=(\lambda,\mu,a, b,c)$ satisfying the inequalities
\begin{eqnarray}
\label{ineq1}&0< \lambda  <1, \, \mu >0,  \,  0 \le c  < b \le 1, \, \lambda\mu \le 1 \, \text{or} \,\,  (1-b)\mu \le 1-c,\\ 
\label{ineq2}&b-b\lambda <a <b-c\lambda ,\\
\label{ineq3}&(1-\lambda)(c- \mu b)  < (1-\mu)a. 
\end{eqnarray}
Set $\displaystyle  \eta = {b-a\over \lambda}$ and  define   $f_{\p}: I \to I$  by the splitted formula
$$
f_{\p}(x) =  \begin{cases} \lambda x + a,   &\text{if} \quad  0\le x < \eta , 
\\
\lambda\mu(x-\eta)+ c,    & \text{if}\quad \eta \le x <1.
\end{cases}
$$
\end{definition}

The left branch of the graph  has slope $\lambda$ and endpoints $(0,a)$ and $(\eta,b)$, while the right branch has slope $\lambda\mu$ and has for origin the point $(\eta,c)$. 
Assumption \eqref{ineq2} ensures that the break point $\eta$ belongs to the open interval $(c,b)$. It implies that $f_{\p}$ has no fixed point. 
Assumption \eqref{ineq3} ensures that $f_{\p}$ is injective when restricted to the interval $J:=[c,b)$. 
Indeed, we have
$$
f_{\p}(c)= \lambda c + a \quad  \text{and}\quad  f_{\p}(b^-)= \lambda\mu(b-\eta)+ c = \lambda\mu b +\mu(a-b) +c,
$$
so that the inequality 
\begin{equation}\label{injectivity}
f_{\p}(c)> f_{\p}(b^-)
\end{equation}
 is equivalent to \eqref{ineq3}.
Equivalently, \eqref{ineq3} ensures that $f_{\p}$ contracts on average in $[c,b)$, i.e.,   
\begin{equation}\label{ineq3'}
\lambda \frac{\eta-c}{b-c}+\lambda\mu \frac{b-\eta}{b-c}<1.
\end{equation}
It remains to show that the image $f_{\p}(I) $ is contained in $I$. Looking at Figure 1, it is sufficient to prove that $f_{\p}(1^-) < 1$. 
We postpone the proof to Section \ref{sec:desc} which makes use of the assumptions,
 $ \lambda \mu \le 1$ or $\displaystyle \mu \le {1-c\over 1-b}$, occurring in \eqref{ineq1}.  

There is no loss of generality assuming $0<\lambda <1$ in Assumption \eqref{ineq1}, since the homeomorphism 
$\displaystyle (x,y)\in (0,1)^2 \mapsto (1-x,1-y) \in (0,1)^2$
exchanges the two branches and at least one slope must be less than $1$ by \eqref{ineq3'}.

\section{Description of the parameter set $\P$}\label{sec:desc}

In order to study the dynamics of $f_{\p}$, we shall basically view the four parameters $\lambda,\mu,b,c$ as fixed and regard the fifth parameter $a $ as a variable. It turns out that $a$ ranges over an interval once $(\lambda,\mu,b,c)$ has been fixed. The description of this interval is the following.

Set 
\begin{equation}\label{definition of Q}
\QQ:=\left\{(\lambda,\mu,b,c)\colon 0<\lambda<1,\,\mu>0,\, 0\leq c<b\leq 1, \,  \lambda\mu  \le 1\,  
\text{or} \, \,\mu(1-b)\le 1-c  \right\}
\end{equation}
and define 
\begin{equation}\label{definition of d}
d_{\lambda,\mu,b,c}:=\begin{cases} b-c\lambda,& \text{if} \quad \lambda\mu<1,\\
\displaystyle \frac{(1-\lambda)(\mu b-c)}{\mu-1},&\text{if} \quad\lambda\mu\geq1.
\end{cases}
\end{equation}
Notice that $\displaystyle b-c\lambda=\frac{(1-\lambda)(\mu b-c)}{\mu-1}$ when $\lambda\mu=1$. 

\begin{lemma}
For every $(\lambda,\mu,b,c)\in\mathcal{Q}$, the inequality $b-b\lambda<d_{\lambda,\mu,b,c}$ holds and  
$$
\mathcal{P} = \left\{(\lambda,\mu,a,b,c)\colon (\lambda,\mu,b,c)\in\mathcal{Q},\, a\in (b-b\lambda,d_{\lambda,\mu,b,c})\right\}.$$
 \end{lemma}
 
\begin{proof}
Given $(\lambda,\mu,b,c)\in\QQ$ it is a simple exercise to show that $b-b\lambda<d_{\lambda,\mu,b,c}$. Let $(\lambda,\mu,a,b,c)\in\P$. By \eqref{ineq2}, we have the inequalities $b-b\lambda <a<b-c\lambda$. Now, we solve inequality \eqref{ineq3} with respect to $a$. We have three cases:
\begin{itemize}
\item When $\mu<1$,  then \eqref{ineq3} reads
$$
a>\frac{(1-\lambda)(c-\mu b)}{1-\mu}=(1-\lambda) b -\frac{(1-\lambda)(b-c)}{1-\mu} , 
$$
which is greater than $b-b\lambda$, since $\displaystyle \frac{(1-\lambda)(b-c)}{1-\mu}>0$.  We conclude in this case that \eqref{ineq2} and \eqref{ineq3} together give
$$
 b-b\lambda < a < b-c\lambda = d_{\lambda ,\mu,b,c}
$$
as required, since $\lambda\mu<1$.
\item When $\mu>1$,  then \eqref{ineq3} reads
$$
a<\frac{(1-\lambda)(\mu b-c)}{\mu-1} = (1-\lambda)b+\frac{(1-\lambda)(b-c)}{\mu-1}.
$$
Thus, \eqref{ineq2} and \eqref{ineq3} together give the interval 
$$
b-b\lambda < a < \min\left\{b-c\lambda,(1-\lambda)b+\frac{(1-\lambda)(b-c)}{\mu-1}\right\}.
$$
In this case, a simple computation shows that
$$
\min\left\{b-c\lambda,(1-\lambda)b+\frac{(1-\lambda)(b-c)}{\mu-1}\right\} = \begin{cases} b-c\lambda,& \text{if}\quad\lambda\mu<1\\
\displaystyle \frac{(1-\lambda)(\mu b-c)}{\mu-1},&\text{if}\quad \lambda\mu\geq1
\end{cases}
.
$$
\item When $\mu=1$, the inequality \eqref{ineq3} reads $(1-\lambda)(c-b) <0$, which is obviously satisfied. We thus get the interval
$$
b-b\lambda <a < b-c\lambda = d_{\lambda,1,b,c}.
$$
\end{itemize}  
Therefore,  we find the interval $b-b\lambda<a<d_{\lambda,\mu,b,c}$ in the three cases.
\end{proof}

The assumptions, $\lambda \mu \le 1 $ or $(1-b)\mu \le 1-c$,  
occurring in the above definition of $\QQ$ are needed to obtain the
\begin{lemma}
For every $\p \in \P$, the inequality  $f_{\p}(1^-) < 1$ holds true.
\end{lemma}

\begin{proof}
We write 
$$
f_{\p}(1^-)= \lambda\mu(1-\eta) + c = \lambda \mu  - \mu b + \mu a +c,
$$
and use the upper bound $a < d_{\lambda,\mu,b,c}$. When $\lambda\mu \le 1$, 
by \eqref{definition of d} we have $d_{\lambda,\mu,b,c}= b-c\lambda$, so that
$$
f_{\p}(1^-) < \lambda \mu  - \mu b + \mu (b-c\lambda) +c = \lambda\mu+(1-\lambda\mu)c = 1-(1-\lambda\mu)(1-c) \le 1, 
$$
as required. When $\lambda\mu > 1$, 
by \eqref{definition of d} we have $d_{\lambda,\mu,b,c}= { (1-\lambda)(\mu b-c)\over \mu-1}$, so that
$$
f_{\p}(1^-) < \lambda \mu  - \mu b + {\mu (1-\lambda)(\mu b-c)\over \mu-1} +c = \lambda\mu+\left(-1 + {\mu(1-\lambda)\over \mu-1}\right)(\mu b -c) =1 + (\lambda\mu-1) \left( 1 -{\mu b -c\over \mu-1}\right).
$$
Notice finally that the factor $\displaystyle 1 -{\mu b -c\over \mu-1}$ is $\le 0$, since $\displaystyle 1 < \mu \le {1-c\over 1-b}$ by assumption, 
 while $\lambda\mu -1 >0$. 

\end{proof}


\begin{remark}\label{rem:P0}
In \cite{LN21},  a 2-interval piecewise affine map depending on a triple of parameters $(\lambda,\mu,\delta)$ was studied.  The parameters $(\lambda,\mu,\delta)$ were assumed to satisfy the inequalities: 
\begin{equation}\label{parametes of LN}
0<\lambda<1,\quad \mu>0\quad\text{and}\quad 1-\lambda < \delta <d_{\lambda,\mu}:=\begin{cases} 1,&\text{if} \quad \lambda\mu<1\\ 
\displaystyle{  \frac{(1-\lambda)\mu}{\mu-1}} ,&\text{if}\quad \lambda\mu\geq 1\end{cases}.
\end{equation}
In  case  $b=1$ and $c=0$, we recover  the description of the set of parameters in \eqref{parametes of LN}.  To see it,  let
\begin{equation}\label{def:P'0}
\P_0:=\{(\lambda,\mu,a,b,c)\in\P\colon b=1,\, c=0\}.
\end{equation}
Taking into account the definition of $\QQ$ in \eqref{definition of Q}, we have $(\lambda,\mu,1,0)\in \QQ$ if and only if $0<\lambda<1$ and $\mu>0$.  Moreover,  given $(\lambda,\mu,1,0)\in \QQ$ we have
$
d_{\lambda,\mu,1,0}=d_{\lambda,\mu}.
$
Indeed, this follows from the identity
\begin{equation}\label{d identity}
d_{\lambda,\mu,b,c}=d_{\lambda,\mu}(b-c)+c(1-\lambda).
\end{equation}
Thus,
$$
\P_0=\{(\lambda,\mu,a,1,0)\colon 0<\lambda<1,\,\mu>0,\,  1-\lambda<a<d_{\lambda,\mu}\}. 
$$
\end{remark}

\section{Rotation number of $f_{\p}$}\label{sec:rotation}

Let $\p \in \P$ be given. Define a real function $F=F_{\p}\colon \R \to \R$ by
$$
F(x)=\begin{cases}
\lambda x+ a+ (1-\lambda)\lfloor x \rfloor,&\text{if} \quad 0\leq \{x\} <\eta\\
\lambda\mu(x-\eta)+c+1+(1-\lambda\mu)\lfloor x \rfloor,&\text{if}\quad \eta\leq \{x\} <1\\
\end{cases}.
$$

Let $f=f_{\p}$.  Clearly,  $F$ is a \textit{lift} of $f$, i.e., for every $x\in\R$,
\begin{enumerate}
\item $F(x+1)=F(x)+1$,
\item $\{F(x)\}=f(\{x\})$.
\end{enumerate}

Notice that $F$ is not always strictly increasing, so we cannot apply immediately the theory of Rhodes and Thompson developed in \cite{RT86} that generalizes the classical theory of the rotation number of Poincar\'e to circle maps having some strictly increasing lift. Nevertheless, we show  in Lemma~\ref{lem:rotation} below,  that $f$ has a well-defined rotation number.

Let
$$
X = \{x\in \R\colon c\leq \{x\} <b\}.
$$
\begin{lemma}\label{lem:prop of F} $F$ satisfies the following properties:
\begin{enumerate}
\item $F(X)\subseteq X$,
\item $F|_X$ is strictly increasing,
\item for every $x\in \R$ there exists $n\geq0$ such that $F^n(x)\in X$.
\end{enumerate}
\end{lemma}

\begin{figure}[ht] 
\begin{tikzpicture}
\draw (0,0) -- (6,0);
\draw (0,0) -- (0,6);
\draw (0,0) -- (6,6);
\draw (3,3) -- (4,3);
\draw (4,3) -- (4,4);
\draw (4,4) -- (3,4);
\draw (3,4) -- (3,3);
\draw[thick]   (3.5,3) -- (6,5);
\draw[thick]  (0,2) -- (3.5,4);
\node   at (-0.3,-0.3) {$0$};
\draw[dashed] (0,4) -- (3,4);
\draw[dashed] (0,3) -- (3,3);
\draw[dashed] (4,4) -- (4,0);
\draw[dashed] (3,3) -- (3,0);
\draw (6,0) -- (6,6);
\node at (3,-0.4) {$c$};
\node at (4,-0.4) {$b$};
\node at (-0.3,4) {$b$};
\node   at (-0.3,2) {$a$};
\node   at (-0.3,6) {$1$};
\node   at (-0.3,3) {$c$};
\node   at (6,-0.4) {$1$};
\draw  (0,6) -- (6,6);
\draw[red] (1, 2.55) -- (1,0);
\draw[red] (1, 2.55) -- (2.55,2.55);
\draw[red] (2.55, 2.5) -- (2.55,3.45);
\draw[red] (2.55, 3.45) -- (3.45,3.45);
\draw[red] (3.45, 3.45) -- (3.45,0);
\draw[blue] (5.5, 0) -- (5.5,4.6);
\draw[blue] (5.5, 4.6) -- (4.6,4.6);
\draw[blue] (4.6, 4.6) -- (4.6,3.9);
\draw[blue] (4.6, 3.9) -- (3.9,3.9);
\draw[blue] (3.9, 3.9) -- (3.9,0);

\end{tikzpicture}
\caption{ Iterated values of $f$ in blue and red}
\label{fig:orbits}
\end{figure}
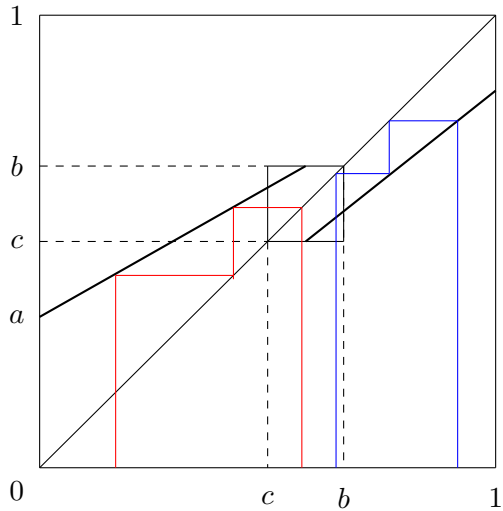

\begin{proof}
Assertion (1) follows from the obvious inclusion $f([c,b))\subseteq [c,b)$, while Assertion (2) follows from the  inequality \eqref{injectivity} by reduction modulo one. Indeed, for any integer $p$, we have
$$
F(b^-+p) =f(b^-)+p+1 < f(c)+p+1 = F(c+p+1).
$$
For Assertion (3),  the smallest value of $n$ depends on the location of $\{ x\}$ on the unit interval. If $\{ x\}\in [c,b)$, we can obviously
choose $n=0$. If $b \le \{ x\} <1$, we have  $c \le f(\{ x\}) < \{ x\}$. Then, if $ f(\{ x\}) < b$, we choose $n=1$. Otherwise, we have 
$b \le f(\{ x\}) <1$ and we iterate, so that $ c \le f^2(\{ x\} ) <f(\{ x\})$. If $ f^2(\{ x \} <b$, we choose $n=2$. Otherwise, we iterate once again and so on. We thus obtain a finite sequence of decreasing points $ \{ x \} > f(\{ x \}) > f^2(\{ x\}) >  \cdots$, which necessarily falls in $[c,b)$ at the end. The argumentation for $ \{ x \} \in [0,c)$ is similar, with now an increasing sequence of iterates.
See Figure \ref{fig:orbits}.
\end{proof}

\begin{lemma}\label{lem:rotation}
$f$ has a well-defined rotation number, i.e., the following limit exists  and is independent of $x\in\R$, 
$$
\rho(f):=\lim_{n\to+\infty}\frac{F^n(x)}{n}\pmod 1.
$$
\end{lemma}

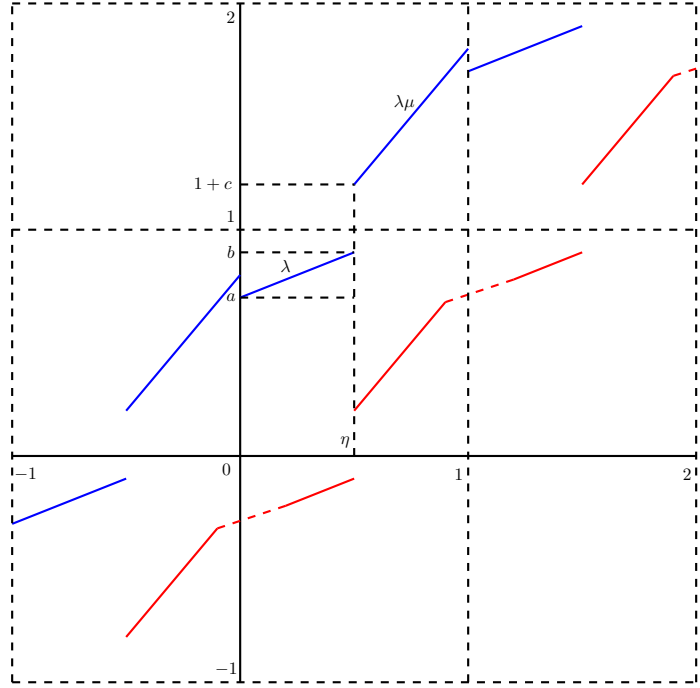
\begin{figure}[ht] 
\begin{tikzpicture}[thick,scale=0.6, every node/.style={scale=0.6}]
\draw (-5,0) -- (10,0);
\draw (0,-5) -- (0,10);
\draw[dashed,red]   (-0.5,3.4-5) -- (1,3.9-5);
\draw[thick,red]  (-2.5,-4) -- (-0.5,3.4-5);
\draw[thick,red]   (1,3.9-5) -- (2.5,4.5-5);
\draw[dashed,red]   (-0.5+5,3.4-5+5) -- (1+5,3.9-5+5);
\draw[thick,red]  (-2.5+5,-4+5) -- (-0.5+5,3.4-5+5);
\draw[thick,red]   (1+5,3.9-5+5) -- (2.5+5,4.5-5+5);

\draw[dashed,red]   (9.5,8.4) -- (10,8.56);
\draw[thick,red]  (-2.5+5+5,-4+5+5) -- (-0.5+5+5,3.4-5+5+5);

\draw[thick,blue]   (0,3.5) -- (2.5,4.5);
\draw[thick,blue]   (-5,-1.5) -- (-2.5,-0.5);
\draw[thick,blue]   (5,8.5) -- (7.5,9.5);
\draw[thick,blue]  (2.5,6) -- (5,9);
\draw[thick,blue]  (-2.5,1) -- (0,4);
\node   at (-0.3,-0.3) {$0$};
\node   at (-0.3,-4.7) {$-1$};
\draw[dashed] (0,6) -- (2.5,6);
\draw[dashed] (-5,5) -- (10,5);
\draw[dashed] (-5,10) -- (10,10);
\draw[dashed] (-5,-5) -- (10,-5);
\draw[dashed] (5,-5) -- (5,10);
\draw[dashed] (-5,-5) -- (-5,10);
\draw[dashed] (10,-5) -- (10,10);
\draw[dashed] (0,4.5) -- (2.5,4.5);
\draw[dashed] (0,3.5) -- (2.5,3.5);
\node at (1,4.2) {$\lambda$};
\node at (3.6,7.8) {$\lambda\mu$};
\node at (-0.2,4.5) {$b$};
\node   at (-0.6,6) {$1+c$};
\node at (2.3,0.3) {$\eta$};
\node   at (-0.2,5.3) {$1$};
\node   at (-0.2,9.7) {$2$};
\node   at (-0.2,3.5) {$a$};
\node   at (4.8,-0.4) {$1$};
\node   at (9.8,-0.4) {$2$};
\node   at (-4.7,-0.4) {$-1$};
\draw[dashed]  (2.5,0) -- (2.5,6);
\end{tikzpicture}
\caption{ A plot of $F$ in blue and of $\bar{F}-1$ in red.}
\label{figure of F}
\end{figure}

\begin{proof}
Using $F$, we define a new piecewise-affine function $\bar{F}\colon \R\to\R$ which equals $F$ inside $X$, but outside $X$ has a graph which is obtained by connecting with a straight segment the end points $(b+p, F(b^-)+p)$ and $(c+p+1,F(c)+p+1)$ for every  $p\in\Z$.  See Figure~\ref{figure of F}. It follows from  Lemma~\ref{lem:prop of F}, that the function $\bar{F}$ is strictly increasing and $\bar{F}-\id$ is $1$-periodic. Therefore, by \cite[Corollary 1]{RT86}, the limit  $\bar{\rho}:= \lim_{n\to+\infty}\frac{\bar{F}^n(x)}{n}\pmod{1}$ exists and is independent of $x$. 
Now,  again by Lemma~\ref{lem:prop of F},  $F(X)\subset X$ and for every $x\in\R$ there is $n_0=n_0(x)\geq 0$ such that $F^{n_0}(x)\in X$.  Since $\bar{F}|_X=F|_X$, we conclude that $F^{n+n_0}(x)=\bar{F}^n(y)$ for every $n\geq0$ where $y=F^{n_0}(x)$. Therefore,
$$
\lim_{n\to+\infty}\frac{F^n(x)}{n}=\lim_{n\to+\infty}\frac{F^{n+n_0}(x)}{n+n_0}
=\lim_{n\to+\infty}\frac{\bar{F}^n(y)}{n+n_0}=\lim_{n\to+\infty}\frac{\bar{F}^n(y)}{n}\frac{n}{n+n_0}
=\bar{\rho}.
$$
\end{proof}

\section{The main results}\label{sec:results}

Given $0<\lambda<1$ and $\mu>0$,  let
$$
r_{\lambda,\mu}:=
\begin{cases}
1, &\text{if}\quad \lambda\mu<1\\
-\frac{\log\lambda}{\log\mu}, &\text{if} \quad\lambda\mu\geq1
\end{cases}.
$$
Notice that $0<r_{\lambda,\mu}\leq 1$.

\begin{definition}\label{def delta}Let $(\lambda,\mu,b,c)\in\QQ$ and assume that $0<\rho<r_{\lambda,\mu}$.  Set
$$
\bdelta(\lambda,\mu,b,c,\rho):=\frac{(1-\lambda)(b+ (\mu b-c)\sigma)}{1+(\mu-1)\sigma },
$$
where 
$$
\sigma =\sigma(\lambda,\mu, \rho) :=  \sum_{k\ge 1} (\lf (k+1)\rho\rf -  \lf k\rho\rfloor )\lambda^k\mu^{ \lf k\rho\rf }  .
$$
\end{definition}

Note that, by \cite[Lemma 8]{LN21},  
$$
\Psi_\rho(\lambda,\mu) = {\lambda\mu\sigma(\lambda,\mu, \rho) \over 1-\lambda},
$$
where 
\begin{equation}\label{eq:Psi}
\Psi_\rho(\lambda,\mu):=\sum_{k\geq1}\sum_{1\leq h\leq k\rho}\lambda^k\mu^h,
\end{equation}
viewed as a power series in the two variables $\lambda$ and $\mu$,  is a Hecke-Mahler series.

 It will be proved in Proposition \ref{lem:delta} that the map $\rho\mapsto \bdelta(\lambda,\mu,b,c,\rho)$ is increasing in the interval $0< \rho < r_{\lambda,\mu}$ and that it has a left discontinuity at  any rational value and is right continuous everywhere.
Our main result is the following.

\begin{theorem}\label{maintheorem}
Let $(\lambda,\mu,b,c)\in\QQ$.  Then, the function $a\mapsto \rho(f_{\lambda,\mu,a,b,c})$ is  continuous and non-decreasing,  mapping the interval $(b-b\lambda,d_{\lambda,\mu,b,c})$ onto the interval $(0,r_{\lambda,\mu})$ and satisfies the following properties:
\begin{enumerate}
\item
For every irrational number $\rho$ with $0 < \rho< r_{\lambda,\mu}$, the rotation number
 $\rho(f_{\lambda,\mu,a,b,c})$ equals $\rho$ if and only if $a=\bdelta(\lambda,\mu,b,c,\rho)$;
\item
Let $p/q$ be a rational number with co-prime positive integers $p<q$ and $0< p/q< r_{\lambda,\mu}$. Then, 
$\rho(f_{\lambda,\mu,a,b,c})=p/q$  if and only if
$$
\bdelta(\lambda,\mu,b,c,(p/q)^-)\leq a\leq \bdelta(\lambda,\mu,b,c,p/q),
$$
where
$$
\bdelta(\lambda,\mu,b,c,(p/q)^-)=\frac{(1-\lambda)(b+ (\mu b-c)\sigma(\lambda,\mu,(p/q)^-)}{1+(\mu-1)\sigma(\lambda,\mu,(p/q)^-) }
$$
and
\begin{align*}
\sigma(\lambda,\mu,(p/q)^-)&=\sigma(\lambda,\mu,p/q))-\frac{\lambda^{q-1}\mu^{p-1}(1-\lambda)}{1-\lambda^q\mu^p},\\
\sigma(\lambda,\mu,p/q)&=\frac{1}{1-\lambda^q\mu^p}\sum_{k=1}^q \left(\left\lfloor (k+1)\frac{p}{q}\right\rfloor - \left\lfloor k\frac{p}{q}\right\rfloor\right)\lambda^k\mu^{\lfloor k\frac{p}{q}\rfloor}.
\end{align*}

\end{enumerate}
\end{theorem}

As a consequence of Theorem~\ref{maintheorem} and a classical result due to Loxton and Van der Poorten  \cite{LoVdPA77,LoVdPB77}, we obtain the following result: 

\begin{corollary}\label{maincorollary} 
Let $(\lambda,\mu,a,b,c)\in\P$.  If $\lambda,\mu,b,c$ and $a$ are algebraic numbers,  then the rotation number $\rho(f_{\lambda, \mu,a,b,c})$ takes a rational value.
\end{corollary}

Given a quintuple $\p=(\lambda,\mu,a,b,c)\in\mathcal{P}$, let $f=f_{\p}$ and 
$$
C=C_{\p}:=\bigcap_{n\geq0}f^n(I)
$$
be the \textit{limit set} of $f$ and, given $x\in I$, 
$$
\omega(x)=\omega(f_{\p},x):=\bigcap_{n\geq0}\overline{\bigcup_{k\geq n} f^k(x)}
$$
be the \textit{$\omega$-limit set of $x$ under $f$.} 

\begin{definition}\label{def:phi} Let $\p=(\lambda,\mu,a,b,c)\in \mathcal{P}$ and  $0< \rho <1$ be such that $\lambda\mu^\rho<1$. 
Let $\phi=\phi_{\p,\rho} \colon \R\to\R$ be defined by
$$
\phi(y)=\lfloor y \rfloor +\frac{a}{1-\lambda}  +\frac{(1-\lambda)(c-\mu b) + a(\mu-1)}{\lambda}\Phi_\rho(\lambda,\mu,-\{y\}),
$$
where
$$
\Phi_{\rho}(\lambda,\mu,y)=\sum_{k\geq 0}\sum_{0\leq l< k\rho+y}\lambda^k\mu^l,
$$
with the convention that a sum indexed by an empty set equals zero. 
\end{definition}
The following result describes the dynamics of $f$ on the limit set $C$.

\begin{theorem}\label{conjugacy}
Let $\p = (\lambda,\mu,a,b,c)\in\P$ and let  $\rho = \rho(f_{\p})$ be the rotation number of $f_{\p}$. Set $\phi=\phi_{\p,\rho}$ where $\phi_{\p,\rho}$ is defined in Definition~\ref{def:phi}. The following holds:
\begin{enumerate}
\item If $\rho$ is irrational, then $C=\phi(I)$, $\overline{C}$ is a Cantor set and $\omega(x)=\overline{C}$ for every $x\in I$. Moreover, $f|_C$ is conjugated by $\phi$ to the rotation $R_\rho\colon x\mapsto x+\rho\pmod{1}$, i.e., the following diagram commutes:

\begin{center}
\begin{tikzcd}
I \arrow{r}{R_\rho} \arrow[swap]{d}{\phi} & I \arrow{d}{\phi} \\%
C \arrow{r}{f}& C\,.
\end{tikzcd}
\end{center}

\item If $\rho=p/q$ is rational with $p$ and $q$ co-prime positive  integers and 
$$
\bdelta(\lambda,\mu,b,c,(p/q)^-)\leq a< \bdelta(\lambda,\mu,b,c,p/q),
$$
then, for every $x\in I$, 
$$
C=\omega(x)=\phi(I)=\{\phi(m/q)\colon 0\leq m<q\}
$$
is a cycle of order $q$ and the following diagram commutes,

\begin{center}
\begin{tikzcd}
\{m/q\colon 0\leq m<q\} \arrow{r}{R_{p/q}} \arrow[swap]{d}{\phi} & \{m/q\colon 0\leq m<q\} \arrow{d}{\phi} \\%
C \arrow{r}{f}& C\,,
\end{tikzcd}
\end{center}
where $R_{p/q}$ denotes the rotation of angle $p/q$. 

\item When $a= \bdelta(\lambda,\mu,b,c,p/q)$, the limit set $C$ is empty, $\phi(I)$ is a finite set containing $b$ and $\omega(x)=\phi(I)$ for every $x\in I$.
\end{enumerate}

\end{theorem}


\section{Reduction of parameters}\label{sec:reduction}

The overall idea of the proof of the results in Section \ref{sec:results} is that, for any $\p\in \P$, the dynamics of the map $f_{\p}$ is determined by its restriction $f_{\p}|_J$ to the invariant interval $J:=[c,b)$. It turns out that, up to an isomorphism,   the dynamics of  $f_{\p}|_J$ has already been studied in \cite{LN21}. See Figure \ref{fig:red} below.

Define the map $\Theta\colon \P\to \R^5$  by
$$
\Theta(\lambda,\mu,a,b,c)= \left(\lambda,\mu, \Delta_{\lambda,b,c}(a),1,0\right),
$$
where 
\begin{equation}\label{eq:Delta}
\Delta_{\lambda,b,c}(a):=\frac{ a -c(1-\lambda)}{b-c}.
\end{equation}
A simple computation shows that
\begin{equation}\label{identity Delta}
 \Delta_{\lambda,b,c}(b-b\lambda)=1-\lambda\quad \text{and}\quad \Delta_{\lambda,b,c}(d_{\lambda,\mu,b,c})=d_{\lambda,\mu}
\end{equation}
for every $(\lambda,\mu,b,c)\in\QQ$. Notice that the second equality follows from the identity \eqref{d identity}.

Let $h\colon J\to I$ be the affine map $x\mapsto (x-c)/(b-c)$.  Recall that $\P_0$ is the set of parameters defined in \eqref{def:P'0} which coincides with  the set of parameters defined in \cite{LN21}.

\begin{lemma}\label{lem:reduction}
$\Theta$ is a projection of $\P$ onto $\P_0$. Moreover, if $\p\in \P$,   then $f_{\p}(J)\subset J$ and the following diagram commutes,
\begin{center}
\begin{tikzcd}
J \arrow{r}{{f_{\p}}|_{J}} \arrow[swap]{d}{h} & J \arrow{d}{h} \\%
I \arrow{r}{f_{\Theta(\p)}}& I\,.
\end{tikzcd}
\end{center}
\end{lemma}

\begin{proof}

Let $\p\in \P$.  By  Remark~\ref{rem:P0},  $\Theta(\p)\in \P_0$ if and only if
$\Delta_{\lambda,b,c}(a)\in (1-\lambda,d_{\lambda,\mu})$. 
Since  $b-b\lambda<a<d_{\lambda,\mu,b,c}$ and the fact that $a\mapsto \Delta_{\lambda,b,c}(a)$ is increasing, we see that $\Delta_{\lambda,b,c}(b-b\lambda)<\Delta_{\lambda,b,c}(a)<\Delta_{\lambda,b,c}(d_{\lambda,\mu,b,c})$. By \eqref{identity Delta}  we conclude that $1-\lambda<\Delta_{\lambda,b,c}(a)<d_{\lambda,\mu}$.
Therefore, $\Theta(\p)\in \P_0$. Because $\Theta(\Theta(\p))=\Theta(\p)$, we see that $\Theta$ is a projection of $\P$ onto $\P_0$. 
Now, taking into account that $\eta = (b-a)/\lambda$, we have that
$$
f_{\p}(\eta)=c<\eta<b=f_{\p}(\eta^-).
$$
Therefore,  $f_{\p}(J)\subset J$ (see Figure~\ref{fig:red}).  Finally, checking that the diagram above commutes is a simple exercise.
\end{proof}

\begin{figure}[h]
\begin{tikzpicture}

\draw[thick,red] (7,0) -- (12,0);
\draw[thick,red] (7,0) -- (7,5);
\draw[thick,red] (12,0) -- (12,5);
\draw[thick,red] (7,5) -- (12,5);

\draw[dashed,red] (1,1) -- (7,0);
\draw[dashed,red] (4.5,1) -- (12,0);
\draw[dashed,red] (4.5,4.5) -- (12,5);
\draw[dashed,red] (1,4.5) -- (7,5);

\draw[dashed,black]   (9.5,0) -- (9.5,5);
\draw[thick,black]   (7,4) -- (9.5,5);
\draw[thick,black]  (9.5,0) -- (12,3.42);
\node   at (6.8,-0.3) {$0$};
\node   at (12,-0.4) {$1$};
\node at (7.9,4.1) {$\lambda$};
\node at (10.6,2.2) {$\lambda\mu$};
\node at (9.5,-0.4)   {$\frac{\eta-c}{b-c}$};
\node   at (6.8,5) {$1$};
\node   at (6.2,4.0) {$\frac{a-c(1-\lambda)}{b-c}$};

\draw (0,0) -- (5,0);
\draw (0,0) -- (0,5);
\draw[thick,red] (1,1) -- (4.5,1);
\draw[thick,red] (4.5,1) -- (4.5,4.5);
\draw[thick,red] (1,1) -- (1,4.5);
\draw[thick,red] (1,4.5) -- (4.5,4.5);

\draw[dotted] (0,0) -- (5,5);
\draw[thick]   (0,3.5) -- (2.5,4.5);
\draw[thick]  (2.5,1) -- (5,4);
\node   at (-0.3,-0.3) {$0$};
\draw[dashed] (0,1) -- (2.5,1);
\draw[dashed] (0,5) -- (5,5);
\draw[dashed] (5,0) -- (5,5);
\draw[dashed] (0,4.5) -- (2.5,4.5);
\node at (0.7,4.1) {$\lambda$};
\node at (3.6,2.8) {$\lambda\mu$};
\node at (-0.3,4.5) {$b$};
\node   at (-0.3,1) {$c$};
\node at (2.5,-0.6) {$\eta = \frac{b-a}{\lambda}$};
\node   at (-0.3,5) {$1$};
\node   at (-0.3,3.5) {$a$};
\node   at (5,-0.4) {$1$};
\draw[dashed]  (2.5,0) -- (2.5,4.5);

\node   at (2.5,-1.5) {$(A)$};
\node   at (9.5,-1.5) {$(B)$};
\end{tikzpicture}
\caption{ (A) Plot of $f_{\p}$ and the square $J^2$ in red.  (B) Zoom of the square $J^2$ using the affine map $h$ and plot of $f_{\Theta(\p)}$.}
\label{fig:red}
\end{figure}
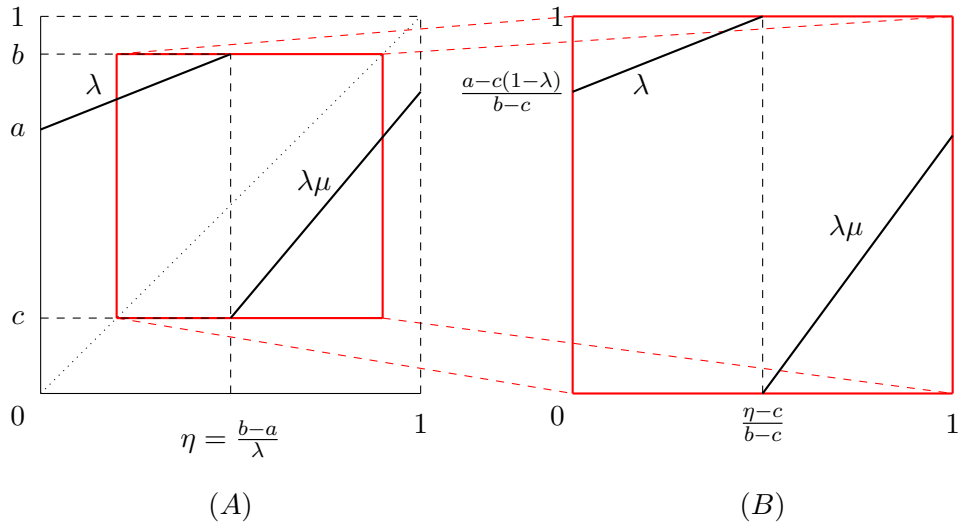


Let $\p=(\lambda,\mu,a,b,c)\in \P$.  Using Lemma~\ref{lem:reduction}, we can reduce the study of the 2-interval piecewise affine map $f_{\p}$ to the study of the 2-interval piecewise affine map $f_{\Theta(\p)}$ that has just 3 parameters. Because the rotation number is invariant by conjugacy we have the following result.

\begin{lemma} \label{lem:6.2}
Let $\p=(\lambda,\mu,a,b,c)\in \P$.  Then $\rho(f_{\p})=\rho(f_{\Theta(\p)})$. 
\end{lemma}

\begin{proof}
Follows from Lemma~\ref{lem:prop of F} and Lemma~\ref{lem:reduction}.
\end{proof}

The map $f_{\Theta(\p)}$ belongs to a family of maps already studied in \cite{LN21} and whose rotation number has been described as a function of the parameters $(\lambda,\mu,\delta)$ where
$
\delta=\Delta_{\lambda,b,c}(a).
$

\section{Properties of $\bdelta$ and $\phi$}\label{sec:properties}

Recalling Definition~\ref{def delta}, the following result describes $\bdelta$ as a function of $\rho$. 

\begin{proposition}\label{lem:delta}
Let $(\lambda,\mu,b,c)\in\QQ$ and assume that $0<\rho<r_{\lambda,\mu}$.  The function $\rho\mapsto \bdelta(\lambda,\mu,b,c,\rho)$ is strictly increasing and right continuous on the interval $(0,r_{\lambda,\mu})$, continuous at every irrational $\rho\in(0,r_{\lambda,\mu})$ and maps the interval $(0,r_{\lambda,\mu})$ inside the interval $(b-b\lambda,d_{\lambda,\mu,b,c})$,  with limit values
$$
\bdelta(\lambda,\mu,b,c,0^+)=b-b\lambda\quad\text{and}\quad \bdelta(\lambda,\mu,b,c,r_{\lambda,\mu}^-)=d_{\lambda,\mu,b,c}.
$$
\end{proposition}

\begin{figure}[ht]
\begin{tikzpicture}
\node at (0,5)
{\includegraphics[width=.50\textwidth]{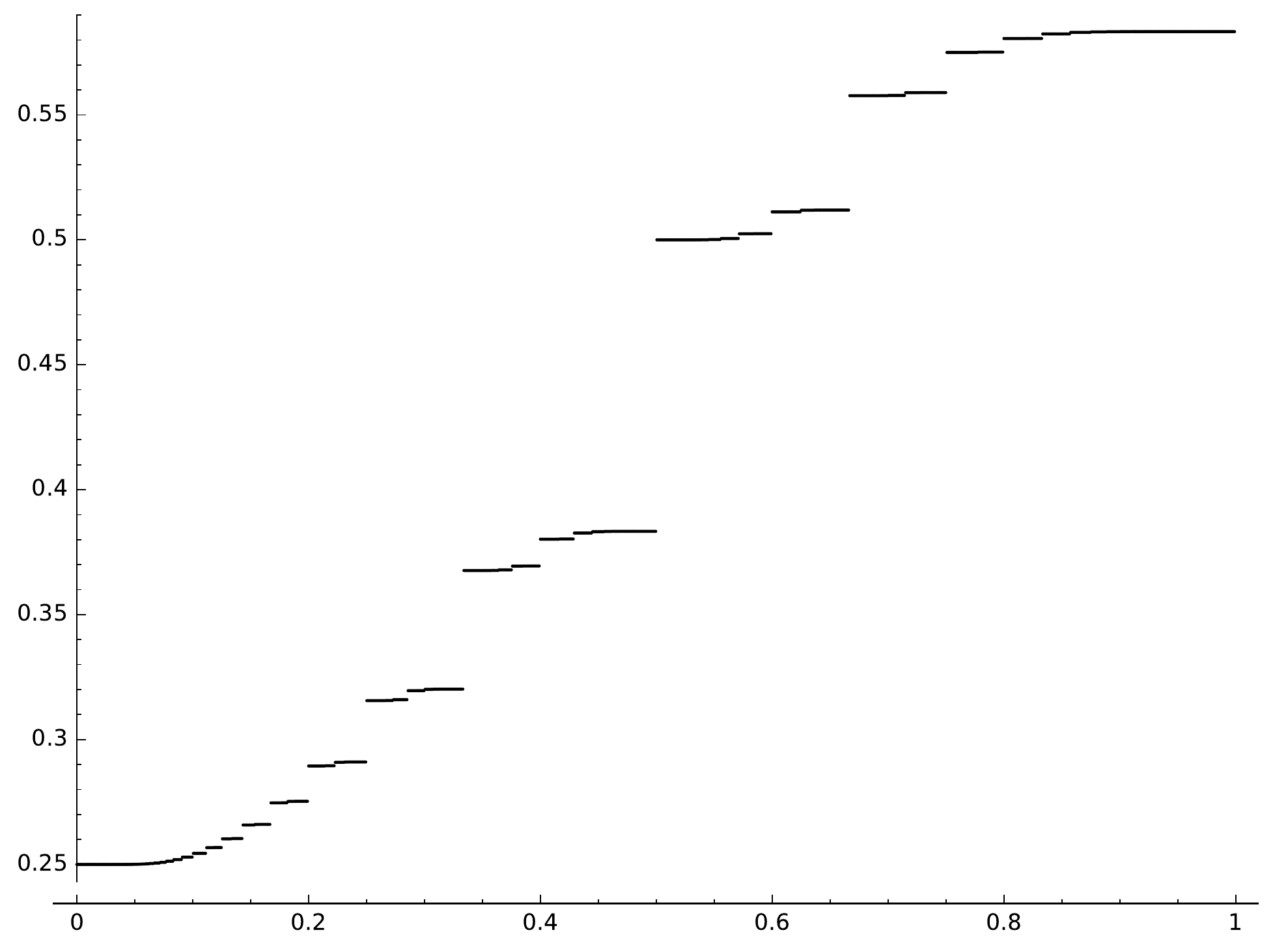}} ;
\end{tikzpicture}
\caption{Plot of the map $\rho \mapsto \bdelta(2/3,1/2,3/4, 1/4, \rho)$  for $\rho \in (0,1)$. The image is contained in the interval $ (b-b\lambda,b-c\lambda)=(1/4,7/12)$.}
\end{figure}

\begin{proof}
Recalling $\Delta_{\lambda,b,c}$ from \eqref{eq:Delta}, notice that
\begin{equation}\label{eq:relation delta and a}
\bdelta(\lambda,\mu,b,c,\rho)=\Delta_{\lambda,b,c}^{-1}(\boldsymbol{\delta}(\lambda,\mu,\rho))= \boldsymbol{\delta}(\lambda,\mu,\rho)(b-c)+c(1-\lambda)
\end{equation}
where 
$$
\boldsymbol{\delta}(\lambda,\mu,\rho)=\frac{(1-\lambda)(1+\mu\sigma(\lambda,\mu,\rho))}{1+(\mu-1)\sigma(\lambda,\mu,\rho)}.
$$
By \cite[Corollary on page  40]{LN21},  for every $0<\lambda<1$ and $\mu>0$, the function $(0,r_{\lambda,\mu})\ni \rho\mapsto \boldsymbol{\delta}(\lambda,\mu,\rho)$ is strictly increasing and its image is contained inside the interval $(\boldsymbol{\delta}(\lambda,\mu,0^+),\boldsymbol{\delta}(\lambda,\mu,r_{\lambda,\mu}^-))$ with limit values 
$$\boldsymbol{\delta}(\lambda,\mu,0^+)=1-\lambda\quad\text{and}\quad \boldsymbol{\delta}(\lambda,\mu,r_{\lambda,\mu}^-)=d_{\lambda,\mu}.$$  Taking into account \eqref{eq:relation delta and a} and the identities \eqref{identity Delta}, Proposition \ref{lem:delta}  follows.
\end{proof}

Recalling Definition~\ref{def:phi},  the following result enumerates some properties of the conjugacy $\phi$. 

\begin{proposition}\label{lem:phiproperties}
Let $\p= (\lambda,\mu,a,b,c)\in \P$ and  let $0<\rho<r_{\lambda,\mu}$ be the rotation number of $f_{\p}$.  
Then the function $\phi=\phi_{\p,\rho}$ satisfies the following properties:
\begin{enumerate}
\item $\phi-\id$ is $1$-periodic,
\item $\phi$ is right continuous and non-decreasing,
\item $\phi$ is strictly increasing 
 if $\rho$ is irrational,
\item $\phi$ is constant on each interval $[\frac{n}{q},\frac{n+1}{q})$, $n\in\Z$ provided $\rho=\frac{p}{q}$ is rational, 
\item $\phi(0)\ge  c $ and $\phi(1^-) \le b$. Moreover equality holds in both inequalities when $\rho$ is irrational, 
\item  for any $y\in \R$, we have the relations 
$$
\lf \phi(y) \rf = \lf y \rf , \quad \{ \phi(y)\} = \phi(\{y\}) \quad \text{and} \quad \phi(y + \rho) = F( \phi(y)). 
$$
\end{enumerate}
\end{proposition}

\begin{figure}[ht]
\begin{tikzpicture}
\node at (0,5)
{\includegraphics[width=.50\textwidth]{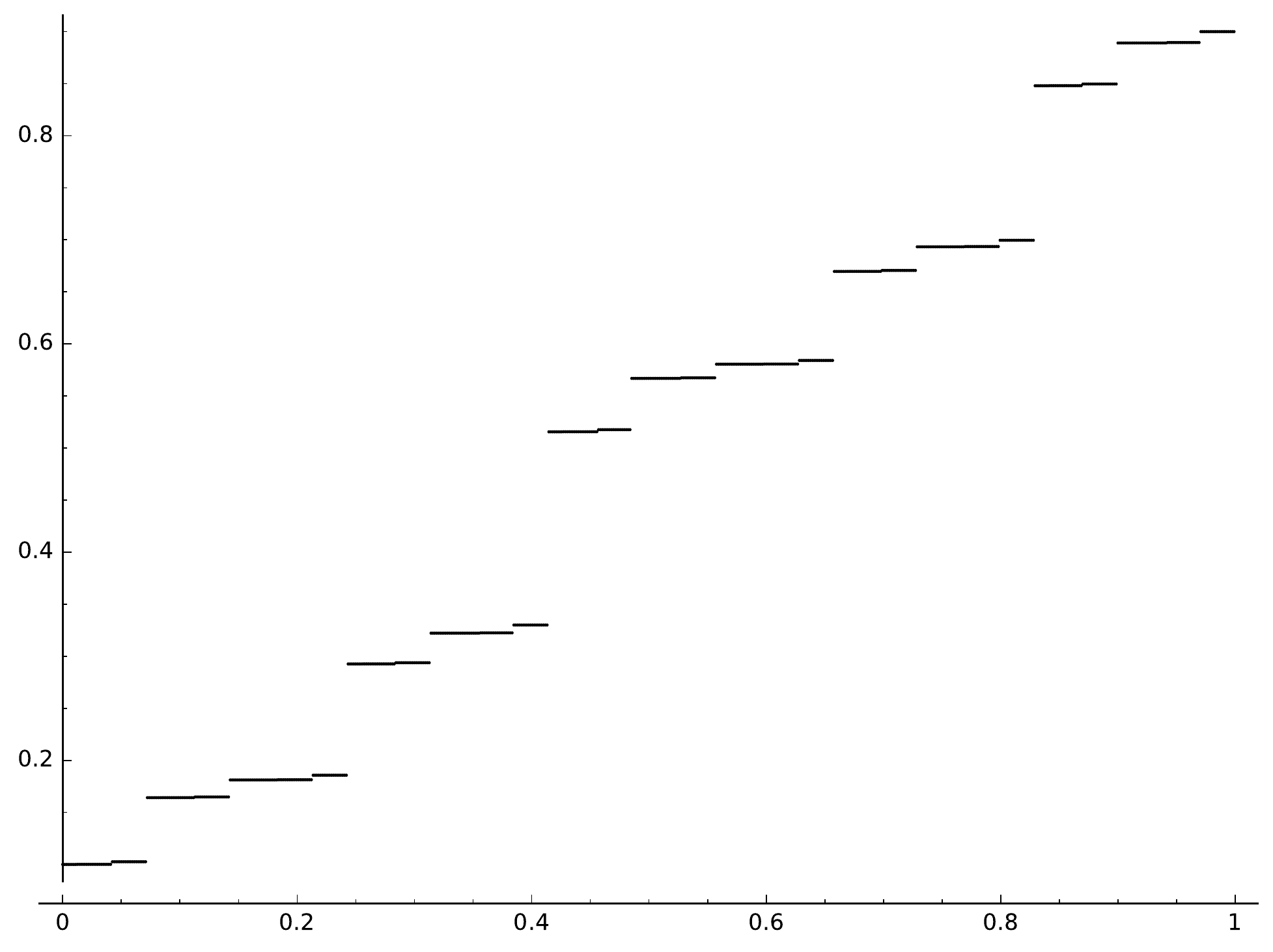}} ;
\end{tikzpicture}
\caption{Graph of the function $\phi_{\p,\rho}(y)$ for the parameters 
$$
\lambda= 0.8,\, \mu= 0.9,\, b=0.9,\, c=0.1, \, a=\bdelta(0.8,0.9,0.9,0.1,\sqrt{2}-1)=0.43557....
$$
and $\rho= \sqrt{2}-1=0.414...$ in the interval $y \in (0,1)$.  
The function $\phi$ increases from $\phi(0^+)=c=0.1$ to $\phi(1^-)=b=0.9$.}
\end{figure}

\begin{proof}
Recall from Section~\ref{sec:reduction} the affine map $h(x)=(x-c)/(b-c)$.  For any $y\in\R$,  a straightforward  computation shows that  
$$
\phi(y)= \lf y \rf + \phi(\{ y \}) = \lf y \rf + h^{-1}(\varphi(\{y\})) = \lf y \rf  + (b-c) \varphi(\{ y \} ) +c,
$$
 where $\varphi\colon \R\to \R$ is the function
$$
\varphi(y)=\lfloor y\rfloor +  \frac{\delta}{1-\lambda}-\frac{\delta-\mu(\lambda+\delta-1)}{\lambda}\Phi_\rho(\lambda,\mu,-\{y\})
$$
as defined in \cite[ Lemma 9]{LN21} and $\delta:=\Delta_{\lambda,b,c}(a)$.   Since $1-\lambda<\delta<d_{\lambda,\mu}$,  all claimed properties follow from \cite{LN21}.
We only  give a detailed proof of the assertion (6), the others being simpler. 
The equalities $ \lf \phi(y) \rf = \lf y \rf $ and $\{\phi(y)\} = \phi(\{y\})$  clearly follow from (1), (2)  and (5). 
It remains to prove the functional equation $\phi(y+ \rho)= F(\phi(y))$. 

  Observe first that $\phi(\{y\}) < \eta$ if $ \{y\} < 1-\rho$, while 
 $\phi(\{y\}) \ge \eta$ if $\{ y \} \ge1-\rho $. Indeed \cite[Lemma 11 and 12]{LN21}, applied to $f_{\Theta(p)}$, shows that the inequalities
 $$
 \varphi(x) <{1-\delta \over \lambda} ={\eta -c \over b-c} = h(\eta) \quad \text{and}  \quad 0 \le x <1
 $$
 hold if and only if $0\le x < 1-\rho$. We make use of the functional equation for $\varphi$, given by \cite[Lemma 14]{LN21}, which reads
 \begin{equation} \label{functeq:varphi}
  \varphi(x+ \rho)= \begin{cases}
\lambda\varphi(x) + \delta + (1-\lambda) \lf x\rf ,  &\text{if} \quad \{ x \}<1- \rho \\
\lambda\mu \varphi(x) + \mu( \delta -1) + 1 + (1-\lambda\mu)\lf x \rf , &\text{if} \quad \{x\} \ge 1-\rho
\end{cases}, 
\end{equation}
 for any $x\in \R$. 
 
 Assume first that $\{ y \} < 1-\rho$. Then, $\lf y+\rho\rf =\lf y\rf $ and $\{y+\rho \} = \{ y\} + \rho$, so that by \eqref{functeq:varphi}
 $$
 \begin{aligned}
 \phi(y+ \rho) & = \lf y \rf +  (b-c)\varphi(\{y\} + \rho)  +c =\lf y\rf + (b-c)(\lambda \varphi(\{y\}) + \delta) + c 
 \\& = \lf y \rf + \lambda(b-c)\varphi(\{y\}) +a+c\lambda  = \lf y \rf + \lambda(\phi(y) -\lf y\rf-c) +a+c\lambda
 \\
 & = \lambda\phi(y) +a + (1-\lambda)\lf y\rf = F(\phi(y)),
 \end{aligned}
 $$
 since $\{\phi(y)\}< \eta$. Assume now that $\{y\} \ge 1- \rho$. In this case, we have $\lf y+\rho\rf =\lf y\rf +1$ and $\{y+\rho \} = \{ y\} + \rho-1$, so that by \eqref{functeq:varphi}
 $$
 \begin{aligned}
 \phi(y+ \rho) & = \lf y \rf + 1+  (b-c)\varphi(\{y\} -1+ \rho)  +c =\lf y\rf +1+ (b-c)( \varphi(\{y\}+\rho) -1) + c 
 \\
 & = \lf y \rf +1+(b-c)\big( \lambda\mu\varphi(\{y\}) + \mu( \delta -1) \big) + c 
 \\ 
  & = \lf y \rf +(b-c) \lambda\mu\varphi(\{y\}) + \mu(a-b+c\lambda) + c +1
\\
 &= \lf y \rf + \lambda\mu(\phi(y) -\lf y\rf-c) +\mu(a-b+c\lambda)+c +1
 \\
 & = \lambda\mu\phi(y)-\lambda\mu \eta +c+1 + (1-\lambda\mu)\lf y\rf = F(\phi(y)),
 \end{aligned}
 $$
 since $\{\phi(y)\}\ge \eta$, replacing  $\delta = {a-c(1-\lambda)\over b-c}$ and $\eta = {b-a\over \lambda}$.
\end{proof}

\section{Proofs of main results}\label{sec:proofs}
\subsection{Proof of Theorem~\ref{maintheorem}}

Let $(\lambda,\mu,b,c)\in\QQ$.  For every $a\in (b-b\lambda,d_{\lambda,\mu,b,a})$,   Lemma~\ref{lem:6.2} asserts that $\rho(f_{\p})=\rho(f_{\Theta(\p)})$ where $\p=(\lambda,\mu,a,b,c)$ and $\Theta(\p)=(\lambda,\mu,\Delta_{\lambda,b,c}(a),1,0)$.
By Lemma~\ref{lem:reduction},  we know that $1-\lambda<\Delta_{\lambda,b,c}(a)<d_{\lambda,\mu}$,  i.e.,  the triple $(\lambda,\mu,\Delta_{\lambda,b,c}(a))$ belongs to the parameter set of \cite{LN21}.  By \cite[Theorem 3]{LN21},  the mapping 
$ \delta \in (1-\lambda,d_{\lambda,\mu}) \mapsto \rho(f_{\lambda,\mu,\delta,1,0})$ is continuous, non-decreasing and its image equals the interval $(0,r_{\lambda,\mu})$.   Since $a\mapsto \Delta_{\lambda,b,c}(a)$ is an increasing affine map sending $(b-b\lambda,d_{\lambda\mu,b,c})$ onto $(1-\lambda,d_{\lambda,\mu})$ and $\rho(f_{\lambda,\mu,a,b,c}) = \rho(f_{\lambda,\mu,\Delta_{\lambda,b,c}(a),1,0})$,  we conclude that the function  $(b-b\lambda,d_{\lambda,\mu,b,a})\ni a\mapsto \rho(f_{\lambda,\mu,a,b,c})$ is continuous, non-decreasing and its image equals the interval $(0,r_{\lambda,\mu})$.   Finally,  to prove properties (1) and (2),  observe that
$$
\boldsymbol{\delta}(\lambda,\mu,\rho) = \Delta_{\lambda,b,c}(\bdelta(\lambda,\mu,b,c,\rho))
$$
where 
$$
\boldsymbol{\delta}(\lambda,\mu,\rho)=\frac{(1-\lambda)(1+\mu\sigma(\lambda,\mu,\rho))}{1+(\mu-1)\sigma(\lambda,\mu,\rho)}
$$
is the function defined in \cite[Definition 2]{LN21}.  Therefore,  (1) and (2) in the statement follow from (1) and (2) of \cite[Theorem 3]{LN21}.
\qed

\subsection{Proof of Corollary~\ref{maincorollary}}
Assume, by contradiction, that the rotation number $\rho=\rho(f_{\lambda,\mu,a,b,c})$ is irrational. Then, by Theorem~\ref{maintheorem} (1), 
$$
a=\frac{(1-\lambda)(b\lambda\mu+(1-\lambda) (\mu b-c)\Psi_\rho(\lambda,\mu))}{\lambda\mu+(\mu-1)(1-\lambda)\Psi_\rho(\lambda,\mu) },
$$
where $\Psi_\rho$ is the Hecke-Mahler series defined in \eqref{eq:Psi}.
As the  coefficients $\lambda$, $\mu$, $b$, $c$ and $a$ are algebraic numbers, we conclude that $\Psi_\rho(\lambda,\mu)$ also takes an algebraic value. However, by \cite[Theorem 10]{LN21}, in this case $\Psi_\rho(\lambda,\mu)$ has to be transcendental, thus it is a contradiction. So $\rho=\rho(f_{\lambda,\mu,a,b,c})$ takes a rational value.
\qed

\subsection{Proof of Theorem~\ref{conjugacy}}
Let $\p=(\lambda,\mu,a,b,c)$ with $(\lambda,\mu,b,c)\in\QQ$ and $a\in (b-b\lambda,d_{\lambda,\mu,b,a})$.  By Lemma~\ref{lem:reduction},  $f_{\p}$ restricted to the interval $J$ is conjugated by the affine map $h:J\to I$ to the map $f_{\Theta(\p)}$ with $\Theta(\p)=(\lambda,\mu,\Delta_{\lambda,b,c}(a),1,0)\in \P_0$ where $\Delta_{\lambda,b,c}$ is defined  in \eqref{eq:Delta}.   Let $g_{\lambda,\mu,\delta}=f_{\Theta(\p)}$ where $\delta:=\Delta_{\lambda,b,c}(a)$.  Because $\Theta(\p)\in\P_0$, we have $1-\lambda<\delta<d_{\lambda,\mu}$ (See Remark~\ref{rem:P0}).  Hence,  the parameters $(\lambda,\mu,\delta)$ belong to the parameter set of \cite{LN21}.  Let $\rho = \rho(f_{\p})$ and notice that $\rho=\rho(g_{\lambda,\mu,\delta})$ by Lemma~\ref{lem:6.2}.  Assume that $\rho$ is irrational.  Applying \cite[Theorem 6]{LN21} to the map $g_{\lambda,\mu,\delta}$, we obtain the following commutative diagram,
\begin{equation*}
\begin{tikzcd}[column sep=huge,row sep=huge]
  J \arrow[d,"f_{\p}|_J"] \arrow[r, "h"] & I \arrow[d,"f_{\Theta(\p)}"]  \arrow[r,leftarrow,"\varphi"] & I \arrow[d,"R_\rho"]\\
  J \arrow[r, "h"] & I \arrow[r,leftarrow,"\varphi"] & I ,
\end{tikzcd}
\end{equation*}
where $\varphi\colon\R\to\R$ is the function
$$
\varphi(y)=\lfloor y\rfloor +  \frac{\delta}{1-\lambda}-\frac{\delta-\mu(\lambda+\delta-1)}{\lambda}\Phi_\rho(\lambda,\mu,-\{y\})
$$
as defined in \cite[Lemma 9]{LN21},  and $\overline{\varphi(I)}$ is a Cantor set which is equal to the $\omega$-limit set $\omega(f_{\Theta(\p)},x)$ of every $x\in I$.   Recall from Section~\ref{sec:reduction} the affine map $h:J\to I$.  For any $y\in\R$,  we have that $\phi(\{y\}) = h^{-1}(\varphi(\{y\}))$. Taking into account that $\overline{h^{-1}(\varphi(I))}$ is also a Cantor set and $\omega(f_{\p},x)=h^{-1}(\omega(f_{\Theta(\p)},h(x))$ for every $x\in J$,  we obtain assertion (1) of Theorem~\ref{conjugacy}.  The assertions (2) and (3) regarding the rational case,  i.e., when $\rho$ is rational, can be deduced in a similar way from  \cite[Theorem 6]{LN21}. Thus, we conclude the proof of the theorem. 
\qed

\section*{Acknowledgement}
JPG was partially supported by the Project CEMAPRE/REM - UIDB/05069/2020 - financed by FCT/MCTES through national funds. 

\bibliographystyle{amsplain}

\vskip 5mm

\centerline{\scshape {\rm Jos\'e Pedro} Gaiv\~ao
 }

{\footnotesize
 \centerline{Department of Mathematics
ISEG, Universidade de Lisboa}
\centerline{Rua do Quelhas 6, 
1200-781, Lisboa
Portugal
 }
 \centerline{jpgaivao@iseg.ulisboa.pt
}}

\medskip

\centerline{\scshape {\rm Michel} Laurent {\rm and } {\rm Arnaldo} Nogueira }

{\footnotesize
 \centerline{Aix Marseille Univ, CNRS, Institut de Math\'ematiques de Marseille}
 \centerline{ 
 163 avenue de Luminy, Case 907,
 13288  MARSEILLE C\'edex 9 France}
 \centerline{michel-julien.laurent@univ-amu.fr and  arnaldo.nogueira@univ-amu.fr}}

\end{document}